\documentclass{amsart}

\usepackage{amsmath}
\usepackage{amssymb}
\usepackage{amscd}
\usepackage{amsthm}
\usepackage[all]{xy}
\usepackage{graphicx}

\theoremstyle{plain}

\newtheorem{thm}{Theorem}[section]
\newtheorem{lem}[thm]{Lemma}

\newtheorem{cor}[thm]{Corollary}

\theoremstyle{definition}

\newtheorem{definition}[thm]{Definition}
\newtheorem{rem}[thm]{Remark}
\newtheorem{ex}[thm]{Example}

\newcommand{\R}{\mathbb{R}}
\newcommand{\Z}{\mathbb{Z}}

\newcommand{\D}{\mathcal{D}}

\newcommand{\abs}[1]{\lvert {#1} \rvert}

\newcommand{\Conf}{\mathrm{Conf}\,}

\newcommand{\fK}[1]{\widetilde{\mathcal{K}}_{#1}}

\newcommand{\Int}{\mathrm{Int}\,}

\newcommand{\K}[1]{\mathcal{K}_{#1}}

\newcommand{\pair}[2]{\langle #1 ,\, #2 \rangle}

\newcommand{\fig}[1]{\raisebox{-0.2\height}{\includegraphics[scale=0.5]{#1}}}

\numberwithin{equation}{section}
\numberwithin{figure}{section}

\title[An integral expression of the first one-cocycle of $\K{3}$]
{An integral expression of the first non-trivial one-cocycle of the space of long knots in $\R^3$}
\author{Keiichi Sakai}
\address{Department of Mathematical Sciences, Shinshu University \\ 3-1-1 Asahi, Matsumoto, Nagano 390-8621, Japan}
\email{ksakai@math.shinshu-u.ac.jp}
\urladdr{http://math.shinshu-u.ac.jp/~ksakai/index.html}
\date{}
\thanks{The author is partially supported by Grant-in-Aid for Young Scientists (B) 21740038, The Sumitomo Foundation,
The Iwanami Fujukai Foundation, and JSPS Research Fellowships for Young Scientists 228006.}

\begin{document}

\begin{abstract}
Our main object of study is a certain degree-one cohomology class of the space $\K{3}$ of long knots in $\R^3$.
We describe this class in terms of graphs and configuration space integrals, showing the vanishing of some anomalous obstructions.
To show that this class is not zero, we integrate it over a cycle studied by Gramain.
As a corollary, we establish a relation between this class and ($\R$-valued) Casson's knot invariant.
These are $\R$-versions of the results which were previously proved by Teiblyum, Turchin and Vassiliev over $\Z /2$ in a different way from ours.
\end{abstract}

\maketitle

\section{Introduction}\label{sec:intro}

A {\em long knot} in $\R^n$ is an embedding $f:\R^1 \hookrightarrow \R^n$ that agrees with the standard inclusion $\iota (t)=(t,0,\dots ,0)$ outside $[-1,1]$.
We denote by $\mathcal{K}_n$ the space of long knots in $\R^n$ equipped with $C^{\infty}$-topology.

In \cite{CCL02} a cochain map $I:\D^* \to \Omega^*_{DR}(\K{n})$ from certain {\em graph complex} $\D^*$ was constructed for $n>3$.
The cocycles of $\K{n}$ corresponding to {\em trivalent graph cocycles} via $I$ generalize an integral expression of finite type invariants for (long) knots in $\R^3$ (see \cite{AltschulerFreidel97, BottTaubes94, Kohno94, Volic05}).
In \cite{K07} the author found a {\em nontrivalent} graph cocycle $\Gamma \in \D^*$ and proved that, when $n>3$ is odd, it gives a non-zero cohomology class $[I(\Gamma )] \in H^{3n-8}_{DR}(\K{n})$.
On the other hand, when $n=3$, some obstructions to $I$ being a cochain map (called {\em anomalous obstructions}; see for example \cite[\S 4.6]{Volic05}) may survive, so even the closedness of $I(\Gamma )$ was not clear.
However, the obstructions for trivalent graph cocycles $X$ (of ``even orders'') in fact vanish \cite{AltschulerFreidel97}, hence the map $I$ still yields closed zero-forms $I(X)$ of $\K{3}$ (they are finite type invariants).
This raises our hope that all the obstructions for any graphs may vanish and hence the map $I$ would be a cochain map even when $n=3$.

In this paper we will show (in Theorem \ref{thm:closed}) that the obstructions for the nontrivalent graph cocycle $\Gamma$ mentioned above also vanish, hence the map $I$ yields the first example of a closed one-form $I(\Gamma )$ of $\K{3}$.
To show that $[I(\Gamma )]\in H^1_{DR}(\K{3})$ is not zero, we will study in part how $I(\Gamma )$ fits into a description of the homotopy type of $\K{3}$ given in \cite{Budney05_2, Budney03, BudneyCohen05}.
It is known that on each component $\K{3}(f)$ that contains $f \in \K{3}$, there exists a one-cycle $G_f$ called the {\em Gramain cycle} \cite{Gramain77, Budney05_2, Turchin06, Vassiliev01}.
The Kronecker pairing gives an isotopy invariant $V:f\mapsto \pair{I(\Gamma )}{G_f}$.
We show in Theorem \ref{thm:Gramain} that $V$ coincides with {\em Casson's knot invariant} $v_2$, which is characterized as the coefficient of $z^2$ in the Alexander-Conway polynomial.
This result will be generalized in Theorem \ref{thm:bracket} for one-cycles obtained by using an action of {\em little two-cubes operad} on the space $\fK{3}$ of {\em framed} long knots \cite{Budney03}.

Closely related results have appeared in \cite{Turchin06,Vassiliev01}, where the $\Z /2$-reduction of a cocycle $v^1_3$ of $\K{n}$ ($n\ge 3$), appearing in the $E_1$-term of Vassiliev's spectral sequence \cite{Vassiliev92}, was studied.
A natural quasi-isomorphism $\D^* \to E_0 \otimes \R$ maps our cocycle $\Gamma$ to $v^1_3$.
In this sense, our results can be seen as ``lifts'' of those in \cite{Turchin06,Vassiliev01} to $\R$.

The invariant $v_2$ can also be interpreted as the linking number of collinearity manifolds \cite{BCSS03}.
Notice that in each formulation (including the one in this paper) the value of $v_2$ is computed by counting some collinearity pairs on the knot.

\section{Construction of a close differential form}\label{s:closed}

\subsection{Configuration space integral}\label{ss:CSI}
We briefly review how we can construct (closed) forms of $\K{n}$ from graphs.
For full details see \cite{CCL02, Volic05}.

Let $X$ be a {\em graph} in a sense of \cite{CCL02, Volic05} (see Figure \ref{fig:Gamma} for examples).
Let $v_{\rm i}$ and $v_{\rm f}$ be the numbers of the {\em interval vertices} (or {\em i-vertices} for short; those on the specified oriented line) and the {\em free} vertices (or {\em f-vertices}; those which are not interval vertices) of $X$, respectively.
With $X$ we associate a configuration space
\[
 C_X := \left\{ \left.
 \begin{array}{l}
  (f;x_1 ,\dots ,x_{v_{\rm i}};x_{v_{\rm i}+1},\dots ,x_{v_{\rm i}+v_{\rm f}}) \\
  \in \K{n} \times \Conf (\R^1, v_{\rm i}) \times \Conf (\R^n, v_{\rm f})
 \end{array}
 \, \right| \,
 \begin{array}{l}
  f(x_i )\ne x_j \text{ for any} \\
  1\le i \le v_{\rm i} <j \le v_{\rm i}+v_{\rm f}
 \end{array}
 \right\}
\]
where $\Conf (M,k):=M^{\times k}\setminus \bigcup_{1\le i<j\le k}\{ x_i=x_j\}$ for a space $M$.

Let $e$ be the number of the edges of $X$.
Define $\omega_X \in \Omega^{(n-1)e}_{DR}(C_X )$ as the wedge of closed $(n-1)$-forms $\varphi^*_{\alpha}{\rm vol}_{S^{n-1}}$, where $\varphi_{\alpha}:C_X \to S^{n-1}$ is the {\em Gauss map}, which assigns a unit vector determined by two points in $\R^n$ corresponding to the vertices adjacent to an edge $\alpha$ of $X$ (for an i-vertex corresponding to $x_i\in\R^1$, we consider the point $f(x_i)\in\R^n$).
Here we assume that ${\rm vol}_{S^{n-1}}$ is ``(anti)symmetric'', namely $i^* {\rm vol}_{S^{n-1}}=(-1)^n {\rm vol}_{S^{n-1}}$ for the antipodal map $i:S^{n-1}\to S^{n-1}$.
Then $I(X)\in\Omega^{(n-1)e-v_{\rm i}-nv_{\rm f}}_{DR}(\K{n})$ is defined by
\[
 I(X):=(\pi_X )_* \omega_X,
\]
the integration along the fiber of the natural fibration $\pi_X:C_X\to \K{n}$.
This fiber is a subspace of $\Conf (\R^1, v_{\rm i}) \times \Conf (\R^n, v_{\rm f})$.
Such integrals converge, since the fiber can be compactified in such a way that the forms $\varphi^*_{\alpha}{\rm vol}_{S^{n-1}}$ are still well-defined on the compactification (see \cite[Proposition 1.1]{BottTaubes94}).
We extend $I$ linearly onto $\D^*$, a cochain complex spanned by graphs.
The differential $\delta$ of $\D^*$ is defined as a signed sum of graphs obtained by ``contracting'' the edges one at a time.

One of the results of \cite{CCL02} states that $I:\D^*\to\Omega^*_{DR}(\K{n})$ is a cochain map if $n>3$.
The proof is outlined as follows.
By the generalized Stokes theorem, $dI(X)=\pm (\pi^{\partial}_X )_*\omega_X$, where $\pi^{\partial}_X$ is the restriction of $\pi_X$ to the codimension one strata of the boundary of the (compactified) fiber of $\pi_X$.
Each codimension one stratum corresponds to a collision of subconfigurations in $C_X$, or equivalently to $A\subset V(X)\cup \{\infty\}$ (here $V(X)$ is the set of vertices of $X$) with a {\em consecutiveness property};
if two i-vertices $p,q$ are in $A$, then all the other i-vertices between $p$ and $q$ are in $A$.
Here ``$\infty\in A$'' means that the points $x_l$ ($l\in A$) escape to infinity.
When $\infty\not\in A$, the interior $\Int\Sigma_A$ of the corresponding stratum $\Sigma_A$ to $A$ is described by the following pullback square
\begin{equation}\label{eq:codim-1-strata}
\begin{split}
 \xymatrix{
        & \Int \Sigma_A \ar[r]\ar[d]\ar[ld]_-{\pi^{\partial_A}_X} & \hat{B}_A \ar[d]^-{\rho_A} \\
  \K{n} & C_{X/X_A} \ar[r]_-{D_A} \ar[l]^-{\pi_{X/X_A}}           & B_A
 }
\end{split}
\end{equation}
Here
\begin{itemize}
\item
 $X_A$ is the maximal subgraph of $X$ with $V(X_A )=A$, and $X/X_A$ is a graph obtained by collapsing the subgraph $X_A$ to a single vertex $v_A$;
\item
 $B_A =S^{n-1}$ if $A$ contains at least one i-vertex, and $B_A =\{ * \}$ otherwise;
\item
 If $A$ consists of i-vertices $i_1 ,\dots ,i_s$ ($s>0$) and f-vertices $i_{s+1},\dots ,i_{s+t}$, then
\[
 \hat{B}_A := \left\{ \left.
 \begin{array}{l}
  (v;(x_{i_1},\dots ,x_{i_s};x_{i_{s+1}},\dots ,x_{i_{s+t}})) \\
  \in S^{n-1} \times \Conf (\R^1 ,s) \times \Conf (\R^n ,t)
 \end{array}
 \, \right| \, \left.
 \begin{array}{l}
  x_{i_p} v \ne x_{i_q}\text{ for any}\\
  1 \le p \le s < q \le s+t
 \end{array}
 \right\} \right\slash \sim
\]
 where $\sim$ is defined as
\begin{multline*}
 (v;(x_{i_1},\dots ,x_{i_s};x_{i_{s+1}},\dots ,x_{i_{s+t}})) \sim \\
 (v;(a(x_{i_1}+r),\dots ,a(x_{i_s}+r);a(x_{i_{s+1}}+rv),\dots ,a(x_{i_{s+t}}+rv)))
\end{multline*}
 for any $a\in \R_{>0}$ and $r\in \R$
 (if $A$ consists only of $t$ f-vertices, then
\[
 \hat{B}_A :=\Conf (\R^n ,t)/(\R^1_{>0} \rtimes \R^n ),
\]
where $\R^1_{>0} \rtimes \R^n$ acts on $\Conf (\R^n ,t)$ by scaling and translation);
\item
 $\rho_A$ is the natural projection;
\item
 when $A$ contains at least one i-vertices, $D_A :C_{X/X_A} \to S^{n-1}$ maps $(f;(x_i))$ to $f'(x_{v_A})/\abs{f'(x_{v_A})}$.
\end{itemize}
We omit the case $\infty\in A$; see \cite[Appendix]{CCL02}.

By properties of fiber integrations and pullbacks, the integration of $\omega_X$ along $\Int\Sigma_A$ can be written as $(\pi_{X/X_A})_* (\omega_{X/X_A} \wedge D^*_A (\rho_A )_* \hat{\omega}_{X_A})$, where $\hat{\omega}_{X_A} \in \Omega^*_{DR}(\hat{B}_A )$ is defined similarly to $\omega_X \in \Omega^*_{DR}(C_X )$.

The stratum $\Sigma_A$ is called {\em principal} if $\abs{A}=2$ {\em hidden} if $\abs{A}\ge 3$, and {\em infinity} if $\infty\in A$.
Since two-point collisions correspond to contractions of edges, we have $dI(X)=I(\delta X)$ modulo the integrations along hidden and infinity faces.
When $n>3$, the hidden/infinity contributions turn out to be zero;
in fact $(\rho_A)_* \hat{\omega}_{X_A}=0$ if $n>3$ and if $A$ is not principal (see \cite[Appendix]{CCL02} or the next Example~\ref{ex:3pts}).
This proves that the map $I$ is a cochain map if $n>3$.

\begin{ex}\label{ex:3pts}
Here we show one example of vanishing of an integration along a hidden face $\Sigma_A$.
Let $X$ be the seventh graph in Figure \ref{fig:Gamma} and $A:=\{ 1,4,5\}$.
Then in \eqref{eq:codim-1-strata}, $B_A=S^{n-1}$ since $A$ contains an i-vertex $1$, and
\[
 \hat{B}_A =\{ (v;x_1;x_4,x_5)\in S^{n-1}\times \R^1\times\Conf (\R^n ,2)\, |\,
 x_1 v\ne x_4,x_5\} /\sim ,
\]
where $(v;x_1;x_4,x_5)\sim (v;a(x_1+r);a(x_4+rv),a(x_5+rv))$ for any $a>0$ and $r\in\R^1$.
The subgraph $X_A$ consists of three vertices $1,4,5$ and three edges $14$, $15$ and $45$.
The open face $\Int\Sigma_A$, where three points $f(x_1),x_4$ and $x_5$ collide with each other, is a hidden face and is described by the square \eqref{eq:codim-1-strata}.
Then the integration of $\omega_X$ along $\Int\Sigma_A$ is $(\pi_{X/X_A})_* (\omega_{X/X_A} \wedge D^*_A (\rho_A )_* \hat{\omega}_{X_A})$, where
\begin{gather*}
 \hat{\omega}_{X_A}=\varphi^*_{14}{\rm vol}_{S^{n-1}}\wedge \varphi^*_{15}{\rm vol}_{S^{n-1}}
 \wedge \varphi^*_{45}{\rm vol}_{S^{n-1}} \in \Omega^{3(n-1)}_{DR}(\hat{B}_A);\\
 \varphi_{1j}:=\frac{x_j-x_1v}{\abs{x_j-x_1v}}\ \ (j=4,5),\quad
 \varphi_{45}:=\frac{x_5-x_4}{\abs{x_5-x_4}}.
\end{gather*}
In this case we can prove that $(\rho_A )_* \hat{\omega}_{X_A}=0$, hence the integration of $\omega_X$ along $\Int\Sigma_A$ vanishes.
Indeed a fiberwise involution $\chi :\hat{B}_A\to\hat{B}_A$ defined by
\[
 \chi (v;x_1;x_4,x_5):= (v;x_1;2x_1v-x_4,2x_1v-x_5)
\]
preserves the orientation of the fiber but $\chi^*\hat{\omega}_{X_A}=-\hat{\omega}_{X_A}$ (here we use that ${\rm vol}_{S^{n-1}}$ is antisymmetric), hence we have $(\rho_A )_* \hat{\omega}_{X_A}=-(\rho_A )_* \hat{\omega}_{X_A}$.\qed
\end{ex}

\subsection{Nontrivalent cocycle}\label{ss:nontrivalent}
It is shown in \cite{CCL02} that, when $n>3$, the induced map $I$ on cohomology restricted to the space of trivalent graph cocycles is injective.
In \cite{K07}, the author gave the first example of a nontrivalent graph cocycle $\Gamma$ (Figure \ref{fig:Gamma}) which also gives a nonzero class $[I(\Gamma )]\in H^{3n-8}_{DR}(\mathcal{K}_n )$ when $n>3$ is odd.
\begin{figure}
\includegraphics{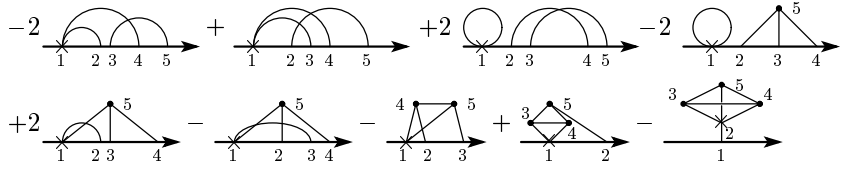}
\caption{A graph cocycle $\Gamma$}
\label{fig:Gamma}
\end{figure}
In Figure~\ref{fig:Gamma}, nontrivalent vertices and trivalent f-vertices are marked by $\times$ and $\bullet$, respectively, and other crossings are not vertices.
Here we say an i-vertex $v$ is trivalent if there is exactly one edge emanating from $v$ other than the specified oriented line.
Each edge $ij$ ($i<j$) is oriented so that $i$ is the initial vertex.

\begin{rem}
An analogous nontrivalent graph cocycle for the space of embeddings $S^1\hookrightarrow\R^n$ for even $n\ge 4$ can be found in \cite{Longoni04}.\qed
\end{rem}

If $n=3$, integrations along some hidden faces (called {\em anomalous contributions}) might survive, and hence the map $I$ might fail to be a cochain map.
However, nonzero anomalous contributions arise from limited hidden faces.

\begin{thm}\label{thm:nonzero_anomalous}
Let $X$ be a graph and $A\subset V(X)\cup\{\infty\}$ be such that $\Sigma_A$ is not principal.
When $n=3$, the integration of $\omega_X$ along $\Sigma_A$ can be nonzero only if the subgraph $X_A$ is trivalent.
\end{thm}

Our main theorem is proved by using Theorem \ref{thm:nonzero_anomalous}.

\begin{thm}\label{thm:closed}
$I(\Gamma ) \in \Omega^1_{DR}(\K{3})$ is a closed form.
\end{thm}

\begin{proof}
We call the nine graphs in Figure \ref{fig:Gamma} $\Gamma_1 ,\dots ,\Gamma_9$ respectively.
The graphs $\Gamma_i$, $i\ne 3,4,9$, do not contain trivalent subgraphs $X_A$ satisfying the consecutive property (see the paragraph just before \eqref{eq:codim-1-strata}).
So $dI(\Gamma_i)=I(d\Gamma_i)$ for $i\ne 3,4,9$ by Theorem~\ref{thm:nonzero_anomalous}.

Possibly the integration of $\omega_{\Gamma_i}$ ($i=3,4,9$) along $\Sigma_A$ ($A:=\{ 2,\dots ,5\}$) might survive, since the corresponding subgraph $X_A$ is trivalent.
However, we can prove $(\rho_A)_*\hat{\omega}_{X_A}=0$ (and hence $dI(\Gamma_i)=I(d\Gamma_i)$) as follows:
$(\rho_A )_* \hat{\omega}_{X_A}=0$ for $\Gamma_3$, because there is a fiberwise free action of $\R_{>0}$ on $\hat{B}_A$ given by translations of $x_2$ and $x_4$ (see \cite[Proposition 4.1]{Volic05}) which preserves $\hat{\omega}_{X_A}$.
Thus $(\rho_A)_*\hat{\omega}_{X_A}=0$ by dimensional reason.
The proof for $\Gamma_4$ has appeared in \cite[page 5271]{BottTaubes94};
$\hat{\omega}_{X_A}=0$ on $\hat{B}_A$ since the image of the Gauss map $\varphi :B_A \to (S^2 )^3$ corresponding to three edges of $X_A$ is of positive codimension.
As for $\Gamma_9$, $(\rho_A )_*\hat{\omega}_{X_A}=0$ follows from $\deg (\rho_A )_*\hat{\omega}_{X_A}=4$ which exceeds $\dim B_A$ (in fact $B_A=\{ *\}$ in this case).
\end{proof}

\begin{proof}[Proof of Theorem~\ref{thm:nonzero_anomalous}]
Let $A$ be a subset of $V(X)$ with $\abs{A}\ge 3$ or $\infty\in A$, and $X_A$ is nontrivalent.
We must show the vanishing of the integrations along the nonprincipal face $\Sigma_A$ of the fiber of $C_X\to\K{3}$.
To do this it is enough to show $(\rho_A )_* \hat{\omega}_{X_A}=0$.
By dimensional arguments (see \cite[(A.2)]{CCL02}) the contributions of infinite faces vanish.
So below we consider the hidden faces $\Sigma_A$ with $\abs{A} \ge 3$.

If $X_A$ has a vertex of valence $\le 2$, then $(\rho_A )_*\hat{\omega}_{X_A}=0$ is proved by dimensional arguments or existence of a fiberwise symmetry of $B_A$ which reverses the orientation of the fiber of $\rho_A :\hat{B}_A\to B_A$ but preserves the integrand $\hat{\omega}_{X_A}$ (like $\chi$ from Example \ref{ex:3pts}, see also \cite[Lemmas A.7-A.9]{CCL02}).

Next, consider the case that there is a vertex of $X_A$ of valence $\ge 4$.
Let $e$, $s$ and $t$ be the numbers of the edges, the i-vertices and the f-vertices of $X_A$ respectively.
Then $\deg \hat{\omega}_{X_A}=2e$ and the dimension of the fiber of $\rho_A$ is $s+3t-k$, where $k=2$ or $4$ according to whether $s>0$ or $s=0$ (see \cite[(A.1)]{CCL02}).
Thus $(\rho_A )_*\hat{\omega}_{X_A}\in\Omega^*_{DR}(B_A)$ is of degree $2e-s-3t+k$.
It is not difficult to see $2e-s-3t>0$ because at least one vertex of $X_A$ is of valence $\ge 4$.
Hence $\deg (\rho_A )_* \hat{\omega}_{X_A}$ exceeds $\dim B_A$ ($=0$ or $2$) and hence $(\rho_A )_* \hat{\omega}_{X_A}=0$.

Thus only the integrations along $\Sigma_A$ with $X_A$ trivalent can survive.
\end{proof}

\begin{rem}
Every finite type invariant $v$ for long knots in $\R^3$ can be written as a sum of $I(\Gamma_v )$ ($\Gamma_v$ is a trivalent graph cocycle) and some ``correction terms'' which kill the contributions of hidden faces corresponding to trivalent subgraphs (see \cite{AltschulerFreidel97, BottTaubes94, Kohno94, Volic05}).
So by Theorem \ref{thm:nonzero_anomalous} the problem whether $I:\D^*\to\Omega^*_{DR}(\K{3})$ is a cochain map or not is equivalent to the problem whether one can eliminate all the correction terms from integral expressions of finite type invariants.\qed
\end{rem}

\section{Evaluation on some cycles}\label{sec:evaluation}
Here we will show that $[I(\Gamma )]\in H^1_{DR}(\K{3})$ restricted to some components of $\K{3}$ is not zero.

We introduce two assumptions to simplify computations.

{\bf Assumption 1}.
The support of (antisymmetric) ${\rm vol}_{S^2}$ is contained in a sufficiently small neighborhood of the poles $(0,0,\pm 1)$ as in \cite{K07}.
So only the configurations with the images of the Gauss maps lying in a neighborhood of $(0,0,\pm 1)$ can nontrivially contribute to various integrals below.
Presumably $[I(\Gamma )]\in H^1_{DR}(\K{3})$ may be independent of choices of ${\rm vol}_{S^2}$ (see \cite[Proposition 4.5]{CCL02}).

{\bf Assumption 2}.
Every long knot in $\R^3$ is contained in $xy$-plane except for over-arc of each crossing, and each over-arc is in $\{ 0\le z\le h \}$ for a sufficiently small $h>0$ so that the projection onto $xy$-plane is a regular diagram of the long knot.

\subsection{The Gramain cycle}\label{ss:Gramain}
For any $f\in\K{3}$, we denote by $\K{3}(f)$ the component of $\K{3}$ which contains $f$.
Regarding $S^1 =\R /2\pi\Z$ and fixing $f$, we define the map $G_f:S^1 \to\K{3}(f)$, called the {\em Gramain cycle}, by $G_f(s)(t):=R(s)f(t)$, where $R(s)\in SO(3)$ is the rotation by the angle $s$ fixing ``long axis'' (the $x$-axis).
$G_f$ generates an infinite cyclic subgroup of $\pi_1(\K{3}(f))$ if $f$ is nontrivial \cite{Gramain77}.
The homology class $[G_f]\in H_1(\K{3}(f))$ is independent of the choice of $f$ in the connected component;
if $f_t\in\K{3}$ ($0\le t\le 1$) is an isotopy connecting $f_0$ and $f_1$, then $G_{f_t}:[0,1]\times S^1\to\K{3}$ gives a homotopy between $G_{f_0}$ and $G_{f_1}$.
Therefore the Kronecker pairing gives an isotopy invariant $V(f):=\pair{I(\Gamma)}{G_f}$ for long knots.

\begin{thm}\label{thm:Gramain}
The invariant $V$ is equal to Casson's knot invariant $v_2$.
\end{thm}

\begin{cor}\label{cor:nonzero1}
$[I(\Gamma )|_{\K{3}(f)}]\in H^1_{DR}(\K{3}(f))$ is not zero if $v_2(f)\ne 0$.\qed
\end{cor}

We will prove two statements which characterize Casson's knot invariant: $V$ is of finite type of order two and $V(3_1)=1$, where $3_1$ is the long trefoil knot.
To do this, we will represent $G_f$ using {\em Browder operation}, as in \cite{K07}.

\subsubsection{Little cubes action}\label{sss:n>3}
Let $\fK{n}$ be the space of {\em framed} long knots in $\R^n$ (embeddings $\tilde{f}:\R^1\times D^{n-1}\hookrightarrow\R^n$ that are standard outside $[-1,1]\times D^{n-1}$).
There is a homotopy equivalence $\Phi :\fK{3}\simeq\K{3}\times\Z$ \cite{Budney03} that maps $\tilde{f}$ to the pair $(\tilde{f}|_{\R^1\times\{(0,0)\}},{\rm fr}\tilde{f})$, where the {\em framing number} ${\rm fr}\tilde{f}$ is defined as the linking number of $\tilde{f}|_{\R^1\times\{(0,0)\}}$ with $\tilde{f}|_{\R^1\times\{(1,0)\}}$.
Since ${\rm fr}\tilde{f}$ is additive under the connected sum, $\Phi$ is a homotopy equivalence of $H$-spaces.
In general, $\fK{n}\simeq\K{n}\times\Omega SO(n-1)$ as $H$-spaces, where $\Omega$ stands for the based loop space functor.

In \cite{Budney03} an action of the {\em little two-cubes operad} on the space $\fK{n}$ was defined.
Its second stage gives a map $S^1\times (\fK{n})^2\to\fK{n}$ up to homotopy, which is given as ``shrinking one knot $f$ and sliding it along another knot $g$ by using the framing, and repeating the same procedure with $f$ and $g$ exchanged'' (see \cite[Figure 2]{Budney03}).
Fixing a generator of $H_1(S^1)$, we obtain the {\em Browder operation} $\lambda :H_p(\fK{n})\otimes H_q (\fK{n})\to H_{p+q+1}(\fK{n})$, which is a graded Lie bracket satisfying the Leibniz rule with respect to the product induced by the connected sum.
The author proved in \cite{K07} that $\pair{I(\Gamma )}{r_* \lambda (e,v)}=1$ when $n>3$ is odd, where $r:\fK{n}\to\K{n}$ is the forgetting map, $e\in H_{n-3}(\fK{n})$ comes from the space of framings, and $v\in H_{2(n-3)}(\fK{n})$ is the first nonzero class of $\K{n}$ represented by a map $(S^{n-3})^{\times 2}\to\K{n}$ (see below).

\subsubsection{The case $n=3$}\label{sss:n=3}
In \cite{K07} the assumption $n>3$ was used only to deduce the closedness of $I(\Gamma )$ from the results of \cite{CCL02}.
The cycles $e$ and $v$ are defined even when $n=3$:
\begin{itemize}
\item
 Under the homotopy equivalence $\fK{3}\simeq\K{3}\times\Z$, the zero-cycle $e$ is given by $(\iota ,1)$ where $\iota$ is the trivial long knot ($\iota (t)=(t,0,0)$ for any $t\in \R^1$).
\item
 The zero-cycle $v=v(T)$ is given by $\sum_{\varepsilon_i =\pm 1}\varepsilon_1\varepsilon_2T_{\varepsilon_1,\varepsilon_2}$, where $T=3_1$ and $T_{\varepsilon_1,\varepsilon_2}$ is $T$ with its crossing $p_i$, for $i=1,2$ changed to be positive if $\varepsilon_i=+1$ and negative if $\varepsilon_i=-1$ (see Figure \ref{fig:trefoil}).
\begin{figure}
\includegraphics{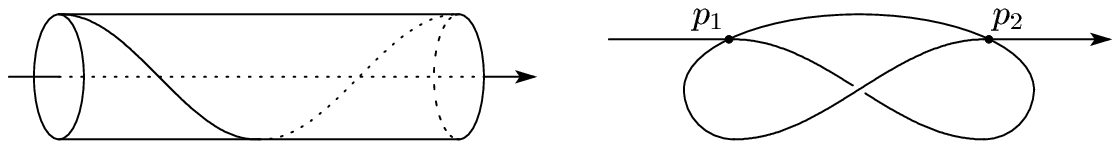}
\caption{The cycles $e$ and $v=v(T)$}
\label{fig:trefoil}
\end{figure}
\end{itemize}
Notice that, for any $f\in\K{3}$ and any pair $(p_1,p_2 )$ of its crossings, an analogous zero-cycle $v=v(f;p_1,p_2)$ can be defined.

Regard $f\in\K{3}$ as a zero-cycle of $\fK{3}$ (with ${\rm fr}f=0$) and consider $r_*\lambda (e,f)$.
During a knot $f$ ``going through'' $e$, $f$ rotates once around $x$-axis.
Thus the one-cycle $r_*\lambda (e,f)$ is homologous to the Gramain cycle $G_f$.
This leads us to the fact that, for $v=v(f;p_1,p_2 )$, the one-cycle $r_*\lambda (e,v)$ is homologous to the sum $\sum_{\varepsilon_i=\pm 1}\varepsilon_1\varepsilon_2G_{f_{\varepsilon_1,\varepsilon_2}}$.
This is why we can apply the method in \cite{K07} to compute
\[
 D^2 V (f) := \sum_{\varepsilon_j=\pm 1}\varepsilon_1 \varepsilon_2 V(f_{\varepsilon_1 ,\varepsilon_2})
 =\sum_{\varepsilon_j=\pm 1}\varepsilon_1 \varepsilon_2 \pair{I(\Gamma )}{G_{f_{\varepsilon_1 ,\varepsilon_2}}}
 =\pair{I(\Gamma )}{r_*\lambda (e,v(f))}.
\]
Recall that our graph cocycle $\Gamma$ is a sum of nine graphs $\Gamma_1,\dots ,\Gamma_9$ (see Figure~\ref{fig:Gamma}).
By Assumption 1, the integration $\pair{I(\Gamma_i)}{G_f}$ can be computed by ``counting'' the configurations with all the images of the Gauss maps corresponding to edges of $\Gamma_i$ being around the poles of $S^2$.
Lemma~\ref{lem:13-24} below was proved in such a way in \cite{K07} when $n>3$.
Since $[v(f)]\in H_0(\K{3}(f))$ is independent of small $h>0$ (see Assumption 2), we may compute $D^2V(f)$ in the limit $h\to 0$.

\begin{definition}
We say that the pair $(p_1,p_2)$ of crossings of $f$ {\em respects the diagram} \fig{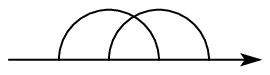} if there exist $t_1 <t_2 <t_3 <t_4$ where $f(t_1 )$ and $f(t_3 )$ correspond to $p_1$, while $f(t_2 )$ and $f(t_4 )$ correspond to $p_2$.
The notion of $(p_1,p_2)$ respecting \fig{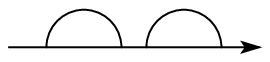} or \fig{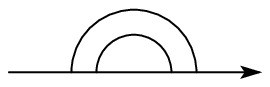} is defined analogously.\qed
\end{definition}

\begin{lem}[\cite{K07}]\label{lem:13-24}
Suppose that $(p_1,p_2)$ respects \fig{13-24.eps}.
Then, in the limit $h \to 0$, $P_i (f):=\sum_{\varepsilon_j =\pm 1}\varepsilon_1 \varepsilon_2 \pair{I(\Gamma_i )}{G_{f_{\varepsilon_1 ,\varepsilon_2}}}$ converges to zero for $i\ne 2$, and $P_2 (f)$ converges to $1$.
Thus $D^2 V(f)=1$.
\end{lem}

\begin{proof}[Outline of proof]
Let $\hat{C}_{\Gamma_i}\to S^1$ be the pullback of $C_{\Gamma_i}\to\K{3}$ via $G_f$, and let $\hat{G}_f:\hat{C}_{\Gamma_i}\to C_{\Gamma_i}$ be the lift of $G_f$.
By the properties of pullbacks and fiber-integrations,
\begin{equation}\label{eq:evaluation}
 P_i (f) = \sum_{\varepsilon_i =\pm 1} \varepsilon_1 \varepsilon_2 \int_{\hat{C}_{\Gamma_i}}
 \hat{G}^*_{f_{\varepsilon_1 ,\varepsilon_2}} \omega_{\Gamma_i}.
\end{equation}
Let $t_1<\dots <t_4$ be such that $f(t_1)$ and $f(t_3)$ correspond to $p_1$, while $f(t_2)$ and $f(t_4)$ correspond to $p_2$.
Define the subspace $C'_{\Gamma_i}\subset\hat{C}_{\Gamma_i}$ as consisting of $(G_f (s);(x_j ))$ ($s\in S^1$) such that, for each $j=1,2$, there is a pair $(l,m)$ of i-vertices of $\Gamma_i$ such that $x_l$ is on the over-arc of $p_j$, $x_m$ is on the under-arc of $p_j$, and there is a sequence of edges in $\Gamma_i$ from $l$ to $m$.

{\em First observation}:
The integration over $\hat{C}_{\Gamma_i}\setminus C'_{\Gamma_i}$ does not essentially contribute to $P_i(f)$ in the limit $h\to 0$.
This is because, over $\hat{C}_{\Gamma_i}\setminus C'_{\Gamma_i}$, the integrals in \eqref{eq:evaluation} are well-defined and continuous even when $h=0$ ($p_j$ becomes a double point), so two terms in $P_i (f)$ corresponding to $\varepsilon_j=\pm 1$ cancel each other.
This implies $\lim_{h\to 0}P_i(f) =0$ for $i=7,8,9$, since $C'_{\Gamma_i} =\emptyset$ if $\sharp\{\text{i-vertices}\}\le 3$.

{\em Second observation}:
Consider the configurations $(x_i)\in C'_{\Gamma_i}$ such that, for any pair $(l,m)$ of i-vertices of $\Gamma_i$ with $x_l$ on the over-arc of $p_j$ and $x_m$ on the under-arc of $p_j$, all the points $x_k$ ($k$ is in a sequence in $\Gamma_i$ from $l$ to $m$) are not near $p_j$.
Such configurations also do not essentially contribute to $P_i(f)$ in the limit $h\to 0$, by the same reason as above.
This implies $\lim_{h\to 0}P_i (f)=0$ for $i=4,5,6$;
the configurations $(x_l)\in C'_{\Gamma_i}$ ($4\le i\le 6$) must be such that the point $x_l\in \R^1$ ($1\le l\le 4$) is near $t_l$.
By the second observation, the ``free point'' $x_5$ must be near $p_1$ or $p_2$.
But then $\omega_{\Gamma_i}=0$, since at least one Gauss map $\varphi_{l5}$ has its image outside the support of ${\rm vol}_{S^2}$ (see Assumption 1).
Thus $\lim_{h\to 0}P_i (f)=0$.

Finally consider $P_i (f)$ for $i=1,2,3$.
For $i=1$ we have $\omega_{\Gamma_i}=0$ over $C'_{\Gamma_i}$, since the Gauss map corresponding to the edge $12$ has its image outside of the support of ${\rm vol}_{S^2}$.
The same reasoning, using the loop edge $11$, shows that $\omega_{\Gamma_3}=0$ over $C'_{\Gamma_i}$.
Only $P_2 (f)$ survives, since the configurations with $x_1$ near $t_1$, $x_2$ near $t_2$, $x_3$ and $x_4$ near $t_3$, and $x_5$ near $t_4$, contribute nontrivially to the integral (see \cite[Lemma~4.6]{K07} for details).
\end{proof}

\begin{lem}\label{lem:other_types}
If $(p_1,p_2)$ respects \fig{12-34.eps} or \fig{14-23.eps}, then $D^2 V(f)=0$.
\end{lem}

\begin{proof}
For $i=4,\dots ,9$, we see in the same way as in Lemma \ref{lem:13-24} that $P_i(f)$ approaches $0$ as $h\to 0$.
That $\lim_{h\to 0}P_i(f)$ for $i=2,3$ and the \fig{14-23.eps}-case for $i=1$ is proved by the first observation in the proof of Lemma~\ref{lem:13-24}.

In the \fig{12-34.eps}-case for $P_1(f)$ over $C'_{\Gamma_1}$, only the configurations with $x_j$ near $t_j$, with $j=1,2,3$, and $x_5$ near $t_4$ may essentially contribute to $P_1(f)$; in this case the edges $12$ and $35$ join the over/under arcs of $p_1$ and $p_2$ respectively.
However, the Gauss map $\varphi_{14}$ cannot have its image in the support of ${\rm vol}_{S^2}$, so $\omega_{\Gamma_1}$ vanishes.
\end{proof}

\begin{proof}[Proof of Theorem~\ref{thm:Gramain}]
For three crossings $(p_1,p_2,p_3)$ of $f\in\K{3}$, consider the third difference
\[
 D^3V(f)
 :=\sum_{\varepsilon_j=\pm 1}\varepsilon_1\varepsilon_2\varepsilon_3 V(f_{\varepsilon_1,\varepsilon_2,\varepsilon_3})
 =D^2V(g_{+1})-D^2V(g_{-1})
\]
where $g_{\pm 1}:=f_{+1,+1,\pm 1}$ and $D^2V(g_{\pm 1})$ are taken with respect to $(p_1,p_2)$.
Since the pair $(p_1,p_2)$ of $g_{+1}$ respects the same diagram as $(p_1,p_2)$ of $g_{-1}$, we have $D^2V(g_{+1})=D^2V(g_{-1})$ by above Lemmas~\ref{lem:13-24}, \ref{lem:other_types}.
Thus $D^3V=0$ and hence $V$ is finite type of order two.
Moreover $V(\iota )=0$ for the trivial long knot $\iota$ since $\K{3}(\iota )$ is contractible \cite{Hatcher83}; therefore $G_{\iota}\sim 0$, and $V(3_1)=1$ by Lemma~\ref{lem:13-24} and $V(\iota )=0$.
These properties uniquely characterize Casson's knot invariant $v_2$.
\end{proof}

\subsection{The Browder operations}\label{subsec:bracket}
We denote a framed long knot corresponding to $(f,k)$ under the equivalence $\fK{3}\simeq\K{3}\times\Z$ by $f^k\in\fK{3}$ (unique up to homotopy).
As mentioned above, the Gramain cycle can be written as $[G_f]=[r_*\lambda (f^k,\iota^1)]$ ($k$ may be arbitrary).
Below we will evaluate $I(\Gamma )$ on more general cycles $r_*\lambda (f^k,g^l)$ of $\K{3}$ for any nontrivial $f,g\in\K{3}$ and $k,l\in\Z$.
This generalizes Theorem~\ref{thm:Gramain}.

\begin{thm}\label{thm:bracket}
We have $\pair{I(\Gamma )}{r_*\lambda (f^k,g^l)}=lv_2(f)+kv_2(g)$ for any $f,g\in\K{3}$ and $k,l\in\Z$.
\end{thm}

\begin{cor}\label{cor:nonzero2}
If at least one of $v_2(f)$ and $v_2(g)$ is not zero, then
\[
 [I(\Gamma )|_{\K{3}(f\sharp g)}]\in H^1_{DR}(\K{3}(f\sharp g))\ne 0,
\]
where $\sharp$ stands for the connected sum.
\end{cor}

\begin{proof}
This is because $r_*\lambda (f^k,g^l)$ is a one-cycle of $\K{3}(f\sharp g)$ for any $k,l\in\Z$.
Since $v_2(f)$ or $v_2(g)$ is not zero, there exist some $k,l$ such that $lv_2(f)+kv_2(g)\ne 0$, so $\pair{I(\Gamma )}{r_*\lambda (f^k,g^l)}\ne 0$ by Theorem~\ref{thm:bracket}.
\end{proof}

\begin{rem}
If $v_2(f)=-v_2(g)$, then $v_2(f\sharp g)=0$ since it is known that $v_2$ is additive under $\sharp$.
Hence we cannot deduce $[I(\Gamma )|_{\K{3}(f\sharp g)}]\ne 0$ from Corollary~\ref{cor:nonzero1}.
Moreover if $v_2(f)=-v_2(g)\ne 0$, then Corollary~\ref{cor:nonzero2} implies $[I(\Gamma )|_{\K{3}(f\sharp g)}]\ne 0$.\qed
\end{rem}

To prove Theorem~\ref{thm:bracket}, first we remark that $f^m\sim f^0\sharp\iota^m$.
Since $\lambda$ satisfies the Leibniz rule, $\lambda (f^k,g^l)$ is homologous to
\[
 \lambda(f^0 ,g^0 )\sharp\iota^{k+l}+\lambda(f^0 ,\iota^l )\sharp g^k+\lambda(\iota^k ,g^0 )\sharp f^l+\lambda(\iota^k,\iota^l)\sharp f^0\sharp g^0.
\]
Since by definition $r_*\lambda (f^k,\iota^m)\sim mG_f$ ($k,m\in\Z$) and $G_{\iota}\sim 0$,
\begin{equation}\label{eq:leibniz}
 r_*\lambda (f^k,g^l)\sim r_*\lambda (f^0,g^0)+lG_f\sharp g+kf\sharp G_g.
\end{equation}

Notice that $\sharp$ makes $\K{3}$ an $H$-space and induces a coproduct $\Delta$ on $H^*_{DR}(\K{3})$.

\begin{lem}\label{lem:coprod}
$\Delta ([I(\Gamma )])=1\otimes [I(\Gamma )]+[I(\Gamma )]\otimes 1\in H^*_{DR}(\K{3})^{\otimes 2}$.
\end{lem}

\begin{proof}
$\D$ also admits $\Delta$ defined as a ``separation'' of the graphs by removing a point from the specified oriented line (see \cite[\S3.2]{CCL05}).
Theorem~6.3 of \cite{CCL05} shows, without using $n>3$, that $(I\otimes I)\Delta (X)=\Delta I(X)$ if $X$ satisfies $dI(X)=I(\delta X)$.

As for our graphs in Figure~\ref{fig:Gamma}, $\Delta\Gamma_i=1\otimes\Gamma_i+\Gamma_i\otimes 1$ ($i\ne 3,4$) and
\[
 \Delta (\Gamma_3-\Gamma_4)=1\otimes (\Gamma_3-\Gamma_4)+(\Gamma_3-\Gamma_4)\otimes 1+\Gamma'\otimes\Gamma''+\Gamma''\otimes\Gamma',
\]
where $\Gamma'$ and $\Gamma''$ are as shown in Figure~\ref{fig:coprod}.
\begin{figure}
\includegraphics{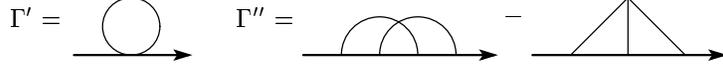}
\caption{Graph cocycles $\Gamma'$ and $\Gamma''$}
\label{fig:coprod}
\end{figure}
Thus
\[
 \Delta I(\Gamma )=1\otimes I(\Gamma )+I(\Gamma )\otimes 1+I(\Gamma')\otimes I(\Gamma'')+I(\Gamma'')\otimes I(\Gamma').
\]
But in fact $\Gamma'=\delta\Gamma_0$ where $\Gamma_0=\fig{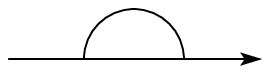}$, and $I(\Gamma')=dI(\Gamma_0 )$ since there is no hidden face in the boundary of the fiber of $\pi_{\Gamma_0}$.
\end{proof}

By \eqref{eq:leibniz}, Lemma~\ref{lem:coprod} and Theorem~\ref{thm:Gramain},
\[
 \pair{I(\Gamma )}{r_*\lambda (f^k,g^l)}=\pair{I(\Gamma )}{r_*\lambda (f^0,g^0)}+lv_2(f)+kv_2(g).
\]
Thus it suffices to prove Theorem~\ref{thm:bracket} in the case $k=l=0$.

\begin{proof}[Proof of Theorem \ref{thm:bracket}]
Fix $g$ and regard $\pair{I(\Gamma )}{r_*\lambda (f^0,g^0)}$ as an invariant $V_g(f)$ of $f$.
We choose two crossings $p_1$ and $p_2$ from the diagram of $f$ in $xy$-plane, and compute
$D^2 V_g (f):=\sum_{\varepsilon_1,\varepsilon_2}\varepsilon_1\varepsilon_2\pair{I(\Gamma )}{r_*\lambda (f^0_{\varepsilon_1,\varepsilon_2},g^0)}$ in the limit $h\to 0$ as in \S\ref{ss:Gramain}.
If this is zero for any $(p_1,p_2)$, then the arguments similar to that in the proof of Theorem~\ref{thm:Gramain} show that $V_g$ is of order two and takes the value zero for the trefoil knot, thus identically $V_g=0$ for any $g$.
This will complete the proof.

We will compute each $P'_i:=\sum_{\varepsilon =\pm 1}\pair{I(\Gamma_i )}{r_*\lambda (f^0_{\varepsilon_1,\varepsilon_2},g^0)}$ ($1\le i\le 9$) in the limit $h\to 0$.
The two observations appearing in the proof of Lemma~\ref{lem:13-24} allow us to conclude $P'_i\to 0$ for $4\le i\le 9$ in the same way as before, so we compute $P'_i$ for $i=1,2,3$ below.
We may concentrate to the integration over $C'_{\Gamma_i}$ by the first observation.
Recall $C'_{\Gamma_i}\subset S^1\times\Conf (\R^1,s)\times\Conf (\R^3,t)$ by definition.
We take $S^1$-parameter $\alpha\in S^1=\R^1/2\pi\Z$ so that $g$ goes through $f$ during $0\le\alpha\le\pi$, and $f$ goes through $g$ during $\pi\le\alpha\le 2\pi$.

First consider the integration over $0\le\alpha\le\pi$.
We may shrink $g$ sufficiently small.
Then the sliding of $g$ through $f$ does not affect the integration, so almost all the integrations converge to zero for the same reasons as in Lemmas~\ref{lem:13-24} and \ref{lem:other_types}.
Only the configurations $(x_i)\in C'_{\Gamma_1}$ with $x_1$ and $x_2$ near $p_1$ may essentially contribute to $P'_1$ when $g$ comes around $p_1$; the form $\varphi^*_{12}{\rm vol}_{S^2}$ may detect the knotting of $g$.
However two terms for $\varepsilon_1=\pm 1$ cancel each other.

Next consider the integration over $\pi\le\alpha\le 2\pi$.
There may be two types of contributions to $P'_i$.
One type comes from the configurations in which all the points on the knot concentrate in a neighborhood of $f$.
Such a contribution depends only on the framing number ${\rm fr}g$ of $g$, not on the global knotting of $g$.
Since ${\rm fr}g^0=0$ here, such configurations do not essentially contribute to $P'_i$.

The other possible contributions arise when $f$ comes near the crossings of $g$.
For example, consider the case that $(p_1,p_2)$ respects \fig{13-24.eps}.
When $f$ comes near a crossing of $g$, a configuration $(x_1,\dots ,x_5)\in C_{\Gamma_1}$ as in Figure~\ref{fig:contribution} is certainly in $C'_{\Gamma_1}$, so it may contribute to $P'_1$.
\begin{figure}
\includegraphics{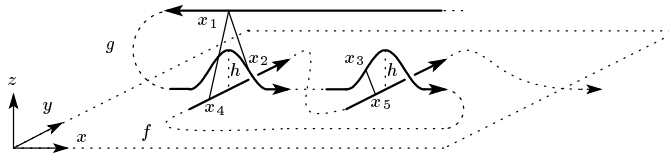}
\caption{When $f$ comes near an under-arc of $g$}
\label{fig:contribution}
\end{figure}
However, such contributions converge to zero in the limit $h\to 0$, because $x_1$ cannot be near $p_1$ (see the second observation in the proof of Lemma~\ref{lem:13-24}).
For $\Gamma_3$, we should take the configuration $(x_1,\dots ,x_5)$ with $x_j$ ($2\le j\le 5$) near $t_{j-1}$ into account; but in this case the Gauss map $\varphi_{11}$ cannot have the image in the support of ${\rm vol}_{S^2}$.
In such ways we can check that all such contributions of $\Gamma_i$ ($i=1,2,3$) can be arbitrarily small.
\end{proof}

\providecommand{\bysame}{\leavevmode\hbox to3em{\hrulefill}\thinspace}
\providecommand{\MR}{\relax\ifhmode\unskip\space\fi MR }
\providecommand{\MRhref}[2]{%
  \href{http://www.ams.org/mathscinet-getitem?mr=#1}{#2}
}
\providecommand{\href}[2]{#2}

\end{document}